\documentclass[
]{elsarticle}

\usepackage{graphicx}
\usepackage{latexsym,amsmath}
\usepackage{amsmath,amsthm,amscd}
\usepackage{amsfonts}
\usepackage[psamsfonts]{amssymb}
\usepackage{enumerate}
\usepackage{mathtools}

\usepackage[final]
{showkeys}

\usepackage{url}

\usepackage{hyperref}

\theoremstyle{plain} 
\newtheorem{theorem}{Theorem}
\newtheorem{corollary}[theorem]{Corollary}

\newtheorem{lemma}
{Lemma}

\newtheorem{proposition}[theorem]{Proposition}

\theoremstyle{definition} 

\theoremstyle{definition} 

\newtheorem*{ex*}{Example}
\theoremstyle{remark} 

\theoremstyle{remark} 
\newtheorem{remark}[theorem]{Remark}
\newtheorem*{remark*}{Remark}

\newcommand{\Ga}{\Gamma}

\newcommand{\la}{\lambda}
\newcommand{\ga}{\gamma}
\newcommand{\de}{\delta}
\newcommand{\be}{\beta}

\renewcommand{\th}{\theta}

\newcommand{\ii}[1]{\,\mathbf{I}\{#1\}} 
\renewcommand{\ii}{\operatorname{I}}

\renewcommand{\P}{\operatorname{\mathsf{P}}} 
\newcommand{\E}{\operatorname{\mathsf{E}}}

\newcommand{\vp}{\varepsilon}

\renewcommand{\le}{\leqslant}
\renewcommand{\ge}{\geqslant}

\journal{Statistics and Probability Letters}

\begin{document}

\begin{frontmatter}


\title{Monotonicity properties of the gamma family of distributions}
%



\author{Iosif Pinelis}

\address{Department of Mathematical Sciences\\
Michigan Technological University\\
Houghton, Michigan 49931, USA\\
E-mail: ipinelis@mtu.edu}

\begin{abstract}
For real $a>0$, let $X_a$ denote a random variable with the gamma distribution with parameters $a$ and $1$. Then $\P(X_a-a>c)$ is increasing in $a$ for each real $c\ge0$; non-increasing in $a$ for each real $c\le-1/3$; and non-monotonic in $a$ for each $c\in(-1/3,0)$. This extends and/or refines certain previously established results.  
\end{abstract}


\end{frontmatter}




\section{Summary and discussion}\label{intro}
For any real $a>0$, let $X_a$ denote a random variable with the gamma distribution with parameters $a$ and $1$, so that for any real $c>-a$ 
\begin{equation*}
	\P(X_a-a>c)=\dfrac{\Ga(a,a+c)}{\Ga(a)}, 
\end{equation*}
where 
\begin{equation}\label{eq:Ga}
	\Ga(a,x)=\int_x^\infty t^{a-1} e^{-t}\,dt 
\end{equation}
for real $x>0$; expression \eqref{eq:Ga} defines the incomplete gamma function. 

There are quite a few bounds on the incomplete gamma function in the literature; see e.g.\ \cite{alzer97,incompl-gamma} and references therein. 

%

The main result of the present paper is 

\begin{theorem}\label{th:1}
The probability $\P(X_a-a>c)$ is 
\emph{
\begin{enumerate}[(I)]
	\item \emph{increasing in real $a>0$ for each real $c\ge0$;}
	\item \emph{decreasing in real $a>-c$ for each real $c\le-1/3$;} 
	\item \emph{non-monotonic in real $a>-c$ for each $c\in(-1/3,0)$.} 
\end{enumerate}
}
\end{theorem}

The terms ``increasing'' and ``decreasing'' are understood in this note in the strict sense, as ``strictly increasing'' and ``strictly decreasing''. 

\begin{remark}\label{rem:a<-c}
Since $\P(X_a-a>c)=1$ for $a\in(0,-c]$, in parts (II) and (III) of Theorem~\ref{th:1} one may replace the condition $a>-c$ by $a>0$, albeit for the price of replacing ``decreasing'' in part (II) by ``non-increasing''. 
\end{remark}

\begin{corollary}\label{cor:1}
For all real $a>0$
\begin{equation*}
	\P(X_a-a>0)<1/2<\P(X_a-a>-1/3).
\end{equation*} 
\end{corollary}

This immediately follows from parts (I) and (II) of Theorem~\ref{th:1} -- because, by the central limit theorem, $\P\big(X_a>a+o(\sqrt a\,)\big)\to1/2$ as $a\to\infty$. In turn, Corollary~\ref{cor:1} immediately implies 

\begin{corollary}\label{cor:2}
For each real $a>0$, the median of $X_a-a$ is in the interval $(-1/3,0)$. 
\end{corollary}

Part (I) of Theorem~\ref{th:1} was previously obtained in  
\cite{chojnacki}, where it was proved by a quite different method, which does not appear to be working for $c<0$. 

Corollary~\ref{cor:2} was previously given in \cite{chen-rubin}. Refinements of this result -- but only for the natural values of $a$ -- were obtained in \cite{adell-jodra,alm,alzer05,choi}. 

Corollary~\ref{cor:1} improves and generalizes the main result of \cite{vietoris}, that \break  $\P(X_n-n>0)<1/2<\P(X_n-n>-1)$ for natural $n$. 


\bigskip

As usual for results on stochastic monotonicity (cf.\ e.g.\ \cite[Section~4]{anderson-samuels67} and \cite[Section~4]{hoeff56}), a straightforward application of Theorem~\ref{th:1} is to statistical testing, as follows. A sample $Y$ is taken from the centered gamma distribution with shape parameter $\th>0$ and scale parameter $1$. We test the null hypothesis $H_0\colon\th=\th_0$ (for some given real $\th_0>0$) versus the alternative hypothesis $H_1\colon\th>\th_0$, using the test $\de(Y):=\ii\{Y>c\}$ with a real critical value $c>0$, where $\ii\{\cdot\}$ denotes the indicator function. Then, according to part (I) of Theorem~\ref{th:1}, the power function $\be_\de$ of the test, given by the formula $\be_\de(\th):=\E_\th\de(Y)=\P_\th(Y>c)=\P(X_\th-\th>c)$ for all real $\th>0$, will be increasing. In particular, it follows that the test $\de$ is unbiased. Part (II) of Theorem~\ref{th:1} can be used 
similarly. 

\section{Proof of Theorem~\ref{th:1}} \label{proof}

\subsection{Proof of part (I) of Theorem~\ref{th:1} (and of part (II) concerning $c\le-1$)} \label{part I}

Take any real $c$ and any real $a>0\vee(-c)$. Then 
\begin{equation}\label{eq:p(a)}
p(a):=p_c(a):=	\P(X_a-a>c)=\frac{\Ga(a,a+c)}{\Ga(a)}=1/\Big(1+\frac{\ga(a,a+c)}{\Ga(a,a+c)}\Big),
\end{equation}
where $\ga(a,a+c):=\Ga(a)-\Ga(a,a+c)$. Note that 
\begin{align*}
	\Ga(a,a+c)&=\int_{a+c}^\infty t^{a-1}e^{-t}\,dt=(a+c)^a\int_1^\infty x^{a-1}e^{-(a+c)x}\,dx, \\
		\ga(a,a+c)&=\int_0^{a+c} t^{a-1}e^{-t}\,dt=(a+c)^a\int_0^1 x^{a-1}e^{-(a+c)x}\,dx. 
\end{align*}
So,
\begin{equation}\label{eq:P=}
	\P(X_a-a>c)=\frac1{1+R(a-1)},
\end{equation}
where
\begin{equation*}
	R(u):=\frac{I(u)}{J(u)},
\end{equation*}

\begin{align*}
	I(u)&:=\int_0^1 f(x)^u e^{-(1+c)x}\,dx=\int_0^1 z^u\, p(z)\,dz, \\
	J(u)&:=\int_1^\infty f(x)^u e^{-(1+c)x}\,dx=\int_0^1 z^u\,q(z)\,dz, 
\end{align*}
\begin{equation*}
	f(x):=xe^{1-x},
\end{equation*}
\begin{equation*}
	p(z):=e^{-(1+c)x_1(z)}x_1'(z)>0,\quad q(z):=-e^{-(1+c)x_2(z)}x_2'(z)>0, 
\end{equation*}
$x_1(z)$ is the only root $x\in(0,1)$ of the equation $f(x)=z$ for $z\in(0,1)$, and $x_2(z)$ is the only root $x\in(1,\infty)$ of the equation $f(x)=z$ for $z\in(0,1)$. 
One might note that for $z\in(0,1)$ we have 
$x_1(z)=-W\left(-z/e\right)$ and $x_2(z)=-W_{-1}\left(-z/e\right)$, where $W$ denotes the principal branch of Lambert's function and $W_{-1}$ denotes its $(-1)$ branch -- see e.g.\ \cite[pages~330--331]{knuth96}. 

It follows that 
\begin{equation}\label{eq:R'=}
\begin{aligned}
	2J(u)^2R'(u)&=2\int_0^1\int_0^1 dx\,dy\,(xy)^u p(x)q(y)(\ln x-\ln y) \\
	&=2\int_0^1\int_0^1 dy\,dx\,(yx)^u p(y)q(x)(\ln y-\ln x) \\
	&=\int_0^1\int_0^1 dy\,dx\,(xy)^u [p(x)q(y)-p(y)q(x)] (\ln x-\ln y)\\
	&=\int_0^1\int_0^1 dy\,dx\,(xy)^u\,p(y)q(y) [r(x)-r(y)](\ln x-\ln y), 
\end{aligned}	
\end{equation}
 where 
\begin{equation*}
	r:=p/q. 
\end{equation*}

Differentiating the identities $f(x_j(z))\equiv z$ for $j=1,2$ in $z\in(0,1)$, we have 
\begin{equation*}
	x_j'(z)=\frac{x_j(z)}{(1-x_j(z)) z}, 
\end{equation*}
which implies 
\begin{equation*}
r'=-Am_c,  
\end{equation*}
where
\begin{equation*}
	A:=A(z):=\frac{e^{(c+1) (x_2-x_1)}(x_2-x_1)x_1}{(1-x_1){}^3 (x_2-1) x_2 z}\in(0,\infty)
\end{equation*}
and 
\begin{equation}\label{eq:m_c=}
\begin{aligned}
	m_c:=m_c(z)&:=c (x_1+x_2-2)+(c+1) (1-x_1 x_2) \\ 
	&=1-x_1 x_2+c (1-x_1)(x_2-1); 
\end{aligned}	
\end{equation}
here we write $x_1$ and $x_2$ in place of $x_1(z)$ and $x_2(z)$, for brevity.

Further, the condition $f(x_1)=f(x_2)$ means that the logarithmic mean of $x_1$ and $x_2$ is $1$; recall that the logarithmic mean of two distinct positive real numbers $x$ and $y$ is defined by the formula 
\begin{equation*}
L(x,y):=\frac{y-x}{\ln y-\ln x}.	
\end{equation*}
Then the arithmetic-logarithmic-geometric mean inequality (see e.g.\ \cite[formula~(4)]{lin74} yields $\sqrt{x_1x_2}<1<(x_1+x_2)/2$, so that 
$
	x_1+x_2-2>0\quad\text{and}\quad 1-x_1 x_2>0. 
$ 
So, $m_c>0$ if $c\ge0$ and $m_c<0$ if $c\le-1$. Since the sign of $r'$ is opposite to that of $m_c$, we see that the function $r$ is decreasing (on $(0,1)$) if $c\ge0$ and increasing if $c\le-1$. Therefore, by \eqref{eq:R'=}, (i) $R'<0$ and hence $R$ is decreasing if $c\ge0$ and (ii) $R'>0$ and hence $R$ is increasing if $c\le-1$. 

Now part (I) of Theorem~\ref{th:1} follows by \eqref{eq:P=} (as well as part (II) concerning $c\le-1$).  

\subsection{Proof of part (II) of Theorem~\ref{th:1}} \label{part II}
Now it is also seen that, to complete the proof of part (II) of Theorem~\ref{th:1}, it suffices to prove Lemma~\ref{lem:1} below; in fact, only the implication (v)$\implies$(i) in Lemma~\ref{lem:1} will be needed for this purpose. 

\begin{lemma}\label{lem:1} Take any real $c$. 
The following statements are equivalent to one another: 
\begin{enumerate}[(i)]
	\item $m_c<0$ on $(0,1)$; 
	\item $c<\dfrac{x_1x_2-1}{(1-x_1)(x_2-1)}$ on $(0,1)$; 
	\item $c<\dfrac{xy-L(x,y)^2}{(L(x,y)-x)(y-L(x,y))}$ whenever $0<x<y<\infty$; 
	\item $c<\la(y):=\dfrac{y-l(y)^2}{(l(y)-1)(y-l(y))}$ for all real $y>1$, where $l(y):=L(1,y)$; 
	\item $c\le-1/3$. 
\end{enumerate}
\end{lemma}

\begin{proof}[Proof of Lemma~\ref{lem:1}]
The equivalence (i)$\iff$(ii) follows immediately from \eqref{eq:m_c=}. The implication (iii)$\implies$(ii) holds because $L(x_1,x_2)=1$, as was noted before. 

To prove the implication (ii)$\implies$(iii), take any $x$ and $y$ such that $0<x<y<\infty$. Let $b:=L(x,y)$. Then $x/b\in(0,1)$, $y/b\in(1,\infty)$, and $f(x/b)=f(y/b)=:z_*$. Then $x/b=x_1(z_*)$ and  $y/b=x_2(z_*)$. So, (ii) will imply 
\begin{equation*}
	c<\dfrac{(x/b)(y/b)-1}{(1-(x/b))((y/b)-1)}=\dfrac{xy-L(x,y)^2}{(L(x,y)-x)(y-L(x,y))}. 
\end{equation*}
This proves the implication (ii)$\implies$(iii). 

The equivalence (iii)$\iff$(iv) follows immediately by homogeneity. 

The remaining equivalence (iv)$\iff$(v) holds by the following lemma. 
\end{proof}

\begin{lemma}\label{lem:2} The function $\la$ defined in Lemma~\ref{lem:1} is increasing on $(1,\infty)$, from $\la(1+)=-1/3$. 
\end{lemma}

The proof Lemma~\ref{lem:2} is based on what was referred to as special l'Hospital-type rule for monotonicity: 

\begin{proposition}\label{prop:}
\emph{[See e.g.\ \cite[Proposition~4.1]{pin06}.]}
Suppose that $-\infty\le A<B\le\infty$. 
Let $f$ and $g$ be differentiable functions defined on the interval $(A,B)$ such that $f(A+)=g(A+)=0$. Suppose further that $g$ and $g'$ do not take on the zero value and do not change their respective signs on $(A,B)$. Finally, suppose that the ``derivative ratio'' $f'/g'$ is increasing on $(A,B)$. Then the ratio $f/g$ is also increasing on $(A,B)$. 
\end{proposition}

\begin{proof}[Proof of Lemma~\ref{lem:2}]
Note that $\la=f/g$, where 
\begin{equation*}
	f(y):=\frac{y \ln^2 y-(y-1)^2}{y},\quad g(y):=\frac{(y-\ln y-1) (y \ln y-y+1)}{y}; 
\end{equation*}
everywhere in this proof, $y$ is an arbitrary real number $>1$. Note also that $f(1+)=g(1+)=0$. Next, here  the ``derivative ratio'' is 
\begin{equation*}
	\frac{f'(y)}{g'(y)}=\frac{f_1(y)}{g_1(y)},
\end{equation*}
where 
\begin{equation*}
	f_1(y):=yf'(y)=\frac1y-y+2\ln y,\quad g_1(y):=yg'(y)=\frac{(y-1)^2\ln y}{y}.  
\end{equation*}
We have $f_1(1+)=g_1(1+)=0$. Next, the ``derivative ratio'' for $f_1/g_1$ is 
\begin{equation*}
	\frac{f_1'(y)}{g_1'(y)}=\frac{f_2(y)}{g_2(y)},
\end{equation*}
where 
\begin{equation*}
	f_2(y):=\frac{y^2}{y^2-1}\,f_1'(y)=\frac{1-y}{1+y},\quad 
	g_2(y):=\frac{y^2}{y^2-1}\,g_1'(y)=\ln y+\frac{y-1}{1+y}.  
\end{equation*}
We have $f_2(1+)=g_2(1+)=0$. 
Further, the ``derivative ratio'' for $f_2/g_2$ is 
\begin{equation*}
	r_3(y):=\frac{f_2'(y)}{g_2'(y)}=-\frac{2y}{1+4y+y^2},
\end{equation*}
whose derivative $2 (y^2-1)/(1 + 4 y + y^2)^2$ is $>0$, for real $y>1$. Applying now Proposition~\ref{prop:} three times, we see that $\la=f/g$ is indeed increasing. Moreover, applying the l'Hospital-type rule for limits three times, we see that $\la(1+)=r_3(1+)=-1/3$. Lemma~\ref{lem:2} is now proved. 
\end{proof}

This completes the proof of parts (I) and (II) of Theorem~\ref{th:1}. 

\begin{remark}
It follows from Lemma~\ref{lem:1} that 
\begin{equation}\label{eq:L<G}
	L(x,y)<\tilde G(x,y):=\sqrt{xy+\tfrac13\,(L(x,y)-x)(y-L(x,y))}
\end{equation}
whenever $0<x<y<\infty$, and the constant factor $\frac13$ here is optimal. This complements the logarithmic-geometric mean inequality $\sqrt{xy}<L(x,y)$ for distinct positive real $x,y$. Also, inequality \eqref{eq:L<G} represents an improvement of the arithmetic-logarithmic mean inequality $L(x,y)<\frac12\,(x+y)$. Indeed, one 
can show that 
$$\tilde G(x,y)<\tfrac12\,(x+y),$$ 
again whenever $0<x<y<\infty$. This can be done 
by a method similar to the one used in the proof of Lemma~\ref{lem:2}, but this time also utilizing the general l'Hospital-type rule for monotonicity given by \cite[Corollary~3.1]{pin06}. 
\end{remark}

\subsection{Proof of part (III) of Theorem~\ref{th:1}} \label{part III}
Take any $c\in(-1/3,0)$. Then, by \eqref{eq:p(a)}, $p((-c)+)=1$, whereas $p(a)<1$ for real $a>-c$. So, $p(a)=\P(X_a-a>c)$ is not increasing in $a$ in any right neighborhood of $0$. 


To complete the proof of part (III) of Theorem~\ref{th:1}, it suffices to show that $p(a+1)>p(a)$ for all large enough $a>0$. Recalling \eqref{eq:p(a)} again and then using integration by parts in the integral expression for $\Ga(a+1,a+1+c)$, we have 
\begin{align*}
	&\Ga(a+1)(p(a+1)-p(a)) \\ 
	&=\Ga(a+1,a+1+c)-a\Ga(a,a+c) \\ 
	&=(a+c+1)^a e^{-a-c-1}-a\int_{a+c}^{a+c+1}x^{a-1}e^{-x}\,dx \\  
	&=(a+c+1)^a \Big(e^{-a-c-1}-a\int_{1-1/(a+c+1)}^1 u^{a-1}e^{-(a+c+1)u}\,du\Big).   
\end{align*}
So, letting $a\to\infty$, 
\begin{equation}\label{eq:b,ep}
	b:=c+1\in(2/3,1),\quad
	\vp:=\frac1{a+b}(\downarrow0), 
\end{equation}
and using the substitution $z=(a+b)v$, we get $v=\vp z$, $a=1/\vp-b$, and 
\begin{align*}
	(p(a+1)-p(a))\,\frac{\Ga(a+1)}{(a+c+1)^a e^{-a-c-1}} 
	&=1-a\int_0^{1/(a+b)} (1-v)^{a-1}e^{(a+b)v}\,dv \\ 
	&=\int_0^1 g(\vp,z)\,dz, 
\end{align*}
where 
\begin{align*}
	g(\vp,z)&:=1-(1-b\vp)(1-z\vp)^{1/\vp-b-1}e^z \\ 
	&=1-\exp\{(z-b+bz-z^2/2)\vp+O(\vp^2)\} \\ 
	&=-(z-b+bz-z^2/2)\vp+O(\vp^2);
\end{align*}
everywhere here, the constant factors in the $O(\cdot)$'s are universal. So, 
\begin{equation*}
	2\int_0^1 g(\vp,z)\,dz=(b-2/3)\vp+O(\vp^2)>0
\end{equation*}
for all small enough $\vp>0$, in view of \eqref{eq:b,ep}. 
Thus, indeed $p(a+1)>p(a)$ for all large enough $a>0$. 

This completes the proof of part (III) of Theorem~\ref{th:1}, and thereby the entire proof of Theorem~\ref{th:1}.

\bibliographystyle{model2-names}
 
%

\bibliography{P:/pCloudSync/mtu_pCloud_02-02-17/bib_files/citations01-09-20}

\end{document}